\documentclass[final,onefignum,onetabnum]{siamart190516}

\usepackage{amsmath,amsfonts,amsopn,amssymb}
\usepackage{graphicx,subfig}
\usepackage{multirow,relsize,makecell}
\usepackage{url}
\ifpdf
  \DeclareGraphicsExtensions{.eps,.pdf,.png,.jpg}
\else
  \DeclareGraphicsExtensions{.eps}
\fi

\numberwithin{theorem}{section}
\newsiamremark{assumption}{Assumption}
\newsiamremark{remark}{Remark}
\crefname{assumption}{Assumption}{Assumptions}
\crefname{remark}{Remark}{Remarks}

\headers{Additive Schwarz with Backtracking}{Jongho Park}

\title{Additive Schwarz Methods for Convex Optimization with Backtracking\thanks{Submitted to the editors DATE.
\funding{This research was supported by Basic Science Research Program through the National Research Foundation of Korea~(NRF) funded by the Ministry of Education~(2019R1A6A1A10073887).
A subset of these results appeared in \textit{Additive Schwarz methods for convex optimization---convergence theory and acceleration} in Proceedings of the 26th International Conference on Domain Decomposition Methods in Science and Engineering, Hong Kong, China, 2020~\cite{DD26}.}
}}

\author{Jongho Park\thanks{Natural Science Research Institute, KAIST, Daejeon 34141, Korea 
  (\email{jongho.park@kaist.ac.kr}, \url{https://sites.google.com/view/jonghopark}).}}

\ifpdf
\hypersetup{
  pdftitle={Additive Schwarz Methods for Convex Optimization with Backtracking},
  pdfauthor={Jongho Park}
}
\fi

\usepackage{algorithm, algorithmic}

\newcommand\gap{\hspace{0.1cm}}

\newcommand{\un}{u^{(n)}}
\newcommand{\unn}{u^{(n+1)}}
\newcommand{\vn}{v^{(n)}}
\newcommand{\vnn}{v^{(n+1)}}
\newcommand{\taun}{\tau^{(n)}}
\newcommand{\taunn}{\tau^{(n+1)}}
\newcommand{\Rw}{R_k^* w_k}

\newcommand{\sumk}{\sum_{k=1}^{N}}

\DeclareMathOperator{\dom}{dom}
\DeclareMathOperator*{\argmin}{\arg\min}

\begin{document}

\maketitle

\begin{abstract}
This paper presents a novel backtracking strategy for additive Schwarz methods for general convex optimization problems as an acceleration scheme.
The proposed backtracking strategy is independent of local solvers, so that it can be applied to any algorithms that can be represented in an abstract framework of additive Schwarz methods.
Allowing for adaptive increasing and decreasing of the step size along the iterations,  the convergence rate of an algorithm is greatly improved.
Improved convergence rate of the algorithm is proven rigorously.
In addition, combining the proposed backtracking strategy with a momentum acceleration technique, we propose a further accelerated additive Schwarz method.
Numerical results for various convex optimization problems that support our theory are presented.
\end{abstract}



\section{Introduction}
\label{Sec:Introduction}
In this paper, we are interested in additive Schwarz methods for a general convex optimization problem 
\begin{equation}
\label{model}
\min_{u \in V} \left\{ E(u) := F(u) + G(u) \right\},
\end{equation}
where $V$ is a reflexive Banach space, $F \colon V \rightarrow \mathbb{R}$ is a Frech\'{e}t differentiable convex function, and $G \colon V \rightarrow \overline{\mathbb{R}}$ is a proper, convex, and lower semicontinuous function that is possibly nonsmooth.
We additionally assume that $E$ is coercive, so that~\cref{model} admits a solution $u^* \in V$.

The importance of studying Schwarz methods arises from both theoretical and computational viewpoints.
It is well-known that various iterative methods such as block relaxation methods, multigrid methods, and domain decomposition methods can be interpreted as Schwarz methods, also known as subspace correction methods.
Studying Schwarz methods can yield a unified understanding of these methods; there have been several notable works on the analysis of domain decomposition and multigrid methods for linear problems in the framework of Schwarz methods~\cite{LWXZ:2008,TW:2005,Xu:1992,XZ:2002}.
The convergence theory of Schwarz methods has been developed for several classes of nonlinear problems as well~\cite{Badea:2006,BK:2012,Park:2020,TX:2002}.
In the computational viewpoint, Schwarz methods are prominent numerical solvers for large-scale problems because they can efficiently utilize massively parallel computer architectures.
There has been plenty of research on Schwarz methods as parallel solvers for large-scale scientific problems of the form~\cref{model}, e.g., nonlinear elliptic problems~\cite{CHW:2020,TX:2002}, variational inequalities~\cite{BTW:2003,Tai:2003,THX:2002}, and mathematical imaging problems~\cite{CTWY:2015,LG:2019,Park:2021a}.

An important concern in the research of Schwarz methods is the acceleration of algorithms.
One of the most elementary relevant results is optimizing the relaxation parameters of Richardson iterations related to the Schwarz alternating method~\cite[section~C.3]{TW:2005}; observing that the Schwarz alternating method for linear elliptic problems can be viewed as a preconditioned Richardson method, one can optimize the relaxation parameters of Richardson iterations to achieve a faster convergence as in~\cite[Lemma~C.5]{TW:2005}.
Moreover, if one replaces Richardson iterations by conjugate gradient iterations with the same preconditioner, an improved algorithm with faster convergence rate can be obtained.
Such an idea of acceleration can be applied to not only linear problems but also nonlinear problems.
There have been some recent works on the acceleration of domain decomposition methods for several kinds of nonlinear problems: nonlinear elliptic problems~\cite{CHW:2020}, variational inequalities~\cite{LP:2021}, and mathematical imaging problems~\cite{LPP:2019,LP:2019b,LZCD:2021}.
In particular, in the author's previous work~\cite{Park:2021b}, an accelerated additive Schwarz method that can be applied to the general convex optimization~\cref{model} was considered.
Noticing that additive Schwarz methods for~\cref{model} can be interpreted as gradient methods~\cite{Park:2020}, acceleration schemes such as momentum~\cite{BT:2009,Nesterov:2013} and adaptive restarting~\cite{OC:2015} that were originally derived for gradient methods in the field of mathematical optimization were adopted.

In this paper, we consider another acceleration strategy called backtracking from the field of mathematical optimization for applications to additive Schwarz methods.
Backtracking was originally considered as a method of line search for step sizes that ensures the global convergence of a gradient method~\cite{Armijo:1966,BT:2009}.
In some recent works on accelerated gradient methods~\cite{CC:2019,Nesterov:2013,SGB:2014}, it was shown both theoretically and numerically that certain backtracking strategies can accelerate the convergence of gradient methods.
Allowing for adaptive increasing and decreasing of the step size along the iterations, backtracking can find a nearly-optimal value for the step size that results in large energy decay, so that fast convergence is achieved. 
Such an acceleration property of backtracking may be considered as a resemblance with the relaxation parameter optimization for Richardson iterations mentioned above.
Hence, as in the case of Richardson iterations for linear problems, one may expect that the convergence rate of additive Schwarz methods for~\cref{model} can be improved if an appropriate backtracking strategy is adopted.
Unfortunately, applying the existing backtracking strategies such as~\cite{CC:2019,Nesterov:2013,SGB:2014} to additive Schwarz methods is not so straightforward.
The existing backtracking strategies require the computation of the underlying distance function of the gradient method.
For usual gradient methods, the underlying distance function is simply the $\ell^2$-norm of the solution space so that such a requirement does not matter.
However, the underlying nonlinear distance function of additive Schwarz methods has a rather complex structure in general~(see~\cref{M}); this aspect makes direct applications of the existing strategies to additive Schwarz methods cumbersome.

This paper proposes a novel backtracking strategy for additive Schwarz methods, which does not rely on the computation of the underlying distance function.
As shown in \cref{Alg:back}, the proposed backtracking strategy does not depend on the computation of the distance function but the computation of the energy functional only.
Hence, the proposed backtracking strategy can be easily implemented for additive Schwarz methods for~\cref{model} with any choices of local solvers.
Acceleration properties of the proposed backtracking strategy can be analyzed mathematically; we present explicit estimates for the convergence rate of the method in terms of some averaged quantity estimated along the iterations.
The proposed backtracking strategy has another interesting feature; since it accelerates the additive Schwarz method in a completely different manner from the momentum acceleration introduced in~\cite{Park:2021b}, both of the momentum acceleration and the proposed backtracking strategy can be applied simultaneously to form a further accelerated method; see \cref{Alg:unified}.
We present numerical results for various convex optimization problems of the form~\cref{model} to verify our theoretical results and highlight the computational efficiency of the proposed accelerated methods.

This paper is organized as follows.
A brief summary of the abstract convergence theory of additive Schwarz methods for convex optimization presented in~\cite{Park:2020} is given in \cref{Sec:ASM}.
In \cref{Sec:Backtracking}, we present and analyze a novel backtracking strategy for additive Schwarz methods as an acceleration scheme.
A fast additive Schwarz method that combines the ideas of the momentum acceleration~\cite{Park:2021b} and the proposed backtracking strategy is proposed in \cref{Sec:Momentum}.
Numerical results for various convex optimization problems are presented in \cref{Sec:Numerical}.
We conclude the paper with remarks in \cref{Sec:Conclusion}.

\section{Additive Schwarz methods}
\label{Sec:ASM}
In this section, we briefly review the abstract framework for additive Schwarz methods for the convex optimization problem~\cref{model} presented in~\cite{Park:2020}.
In what follows, an index $k$ runs from $1$ to $N$.
Let $V_k$ be a reflexive Banach space and let $R_k^* \colon V_k \rightarrow V$ be a bounded linear operator such that
\begin{equation*}
V = \sumk R_k^* V_k
\end{equation*}
and its adjoint $R_k \colon V^* \rightarrow V_k^*$ is surjective.
For the sake of describing local problems, we define $d_k \colon V_k \times V \rightarrow \overline{\mathbb{R}}$ and $G_k \colon V_k \times V \rightarrow \overline{\mathbb{R}}$ as functionals defined on $V_k \times V$, which are proper, convex, and lower semicontinuous with respect to their first arguments.
Local problems have the following general form:
\begin{equation}
\label{local}
\min_{w_k \in V_k} \left\{ F(v) + \langle F'(v), \Rw \rangle
 + \omega d_k (w_k, v) + G_k ( w_k, v) \right\}, \\
\end{equation}
where $v \in V$ and $\omega > 0$.
If we set
\vspace{0cm}
\begin{subequations}
\label{exact_local}
\begin{equation}
d_k (w_k, v) = D_F (v + \Rw , v), \quad
G_k (w_k, v) = G(v + \Rw ), \quad \omega = 1
\end{equation}
in~\cref{local}, then the minimization problem is reduced to
\begin{equation}
\min_{w_k \in V_k} E(v + \Rw ),
\end{equation}
\end{subequations}
which is the case of exact local problems.
Here $D_F$ denotes the Bregman distance
\begin{equation*}
D_F (u,v) = F(u) - F(v) - \left< F'(v) , u-v \right>, \quad u,v \in V.
\end{equation*}
We note that other choices of $d_k$ and $G_k$, i.e., cases of inexact local problems, include various existing numerical methods such as block coordinate descent methods~\cite{BT:2013} and constraint decomposition methods~\cite{CTWY:2015,Tai:2003}; see~\cite[section~6.4]{Park:2020} for details.

The plain additive Schwarz method for~\cref{model} is presented in \cref{Alg:ASM}.
Constants $\tau_0$ and $\omega_0$ in \cref{Alg:ASM} will be given in \cref{Ass:convex,Ass:local}, respectively.
Note that $\dom G$ denotes the effective domain of $G$, i.e.,
\begin{equation*}
\dom G = \left\{ v \in V : G(v) < \infty \right\}.
\end{equation*}

\begin{algorithm}[]
\caption{Additive Schwarz method for~\cref{model}}
\begin{algorithmic}[]
\label{Alg:ASM}
\STATE Choose $u^{(0)} \in \dom G$, $\tau \in (0, \tau_0 ]$, and $\omega \geq \omega_0$.
\FOR{$n=0,1,2,\dots$}
\item \vspace{-0.5cm} \begin{equation*}
\resizebox{0.9\hsize}{!}{ $\displaystyle \begin{split}
w_k^{(n+1)} &\in \argmin_{w_k \in V_k} \left\{ F(\un) + \langle F'(\un), \Rw \rangle + \omega d_k (w_k, \un) + G_k (w_k, \un) \right\}, \gap 1 \leq k \leq N, \\
\unn &= \un + \tau \sumk \Rw^{(n+1)}
\end{split}$}
\end{equation*}
\vspace{-0.4cm}
\ENDFOR
\end{algorithmic}
\end{algorithm}

Note that $u^{(0)} \in \dom G$ implies $E(u^{(0)}) < \infty$.
In what follows, we fix $u^{(0)} \in \dom G$ and define a convex subset $K_0$ of $\dom G$ by
\begin{equation}
\label{K0}
K_0 = \left\{ u \in V : E(u) \leq E(u^{(0)}) \right\}.
\end{equation}
Since $K_0$ is bounded, there exists a constant $R_0 > 0$ such that
\begin{equation}
\label{R0}
K_0 \subseteq \left\{ u \in V : \| u - u^* \| \leq R_0 \right\}.
\end{equation}
In addition, we define
\begin{equation}
\label{Ktau}
K_{\tau} = \left\{ \frac{1}{\tau} u - \left( \frac{1}{\tau} - 1 \right) v : u, v \in K_0 \right\}
\end{equation}
for $\tau > 0$.

An important observation made in~\cite[Lemma~4.5]{Park:2020} is that \cref{Alg:ASM} can be interpreted as a kind of a gradient method equipped with a nonlinear distance function~\cite{Teboulle:2018}.
A rigorous statement is presented in the following.

\begin{lemma}[generalized additive Schwarz lemma]
\label{Lem:ASM}
For $v \in V$ and $\tau, \omega > 0$, we define
\begin{equation*}
\tilde{v} = v + \tau \sumk R_k^* \tilde{w}_k,
\end{equation*}
where
\begin{equation*}
\tilde{w}_k \in \argmin_{w_k \in V_k} \left\{ F(v) + \langle F'(v), \Rw \rangle + \omega d_k (w_k, v) + G_k (w_k, v) \right\}, \quad 1 \leq k \leq N.
\end{equation*}
Then we have
\begin{equation*}
\tilde{v} \in \argmin_{u \in V} \left\{ F(v) + \langle F'(v), u - v \rangle + M_{\tau, \omega} (u, v) \right\},
\end{equation*}
where the functional $M_{\tau, \omega} \colon V \times V \rightarrow \overline{\mathbb{R}}$ is given by
\begin{equation}
\label{M}
\begin{split}
M_{\tau, \omega} (u, v) &= \tau \inf \left\{ \sumk \left( \omega d_k + G_k \right) (w_k, v)  : u-v = \tau \sumk \Rw, \gap w_k \in V_k \right\} \\
 &\quad+ \left(1- \tau N \right) G( v), \quad u,v \in V.
\end{split}
\end{equation}
\end{lemma}

A fruitful consequence of \cref{Lem:ASM} is an abstract convergence theory of additive Schwarz methods for convex optimization~\cite{Park:2020} that directly generalizes the classical theory for linear problems~\cite[Chapter~2]{TW:2005}.
The following three conditions are considered in the convergence theory: stable decomposition, strengthened convexity, and local stability~(cf.~\cite[Assumptions~2.2 to~2.4]{TW:2005}).

\begin{assumption}[stable decomposition]
\label{Ass:stable}
There exists a constant $q > 1$ such that for any bounded and convex subset $K$ of $V$, the following holds:
for any $u, v \in K \cap \dom G$, there exists $w_k \in V_k$, $1\leq k \leq N$, with $u-v = \sumk \Rw$, such that
\begin{equation*}
\sumk d_k (w_k, v) \leq \frac{C_{0,K}^q}{q} \| u-v \|^q , \quad
\sumk G_k ( w_k, v) \leq G \left( u \right) + (N-1) G(v),
\end{equation*}
where $C_{0, K}$ is a positive constant depending on $K$.
\end{assumption}

\begin{assumption}[strengthened convexity]
\label{Ass:convex}
There exists a constant $\tau_0 \in (0, 1]$ which satisfies the following:
for any $v \in V$, $w_k \in V_k$, $1 \leq k \leq N$, and $\tau \in (0, \tau_0 ]$, we have
\begin{equation*}
\left( 1 - \tau N \right) E(v) + \tau \sumk E (v + \Rw) \geq E \left( v + \tau \sumk \Rw \right).
\end{equation*}
\end{assumption}

\begin{assumption}[local stability]
\label{Ass:local}
There exists a constant $\omega_0 > 0$ which satisfies the following:
for any $v \in \dom G$, and $w_k \in V_k$, $1 \leq k \leq N$, we have
\begin{equation*}
D_F ( v + \Rw, v ) \leq \omega_0 d_k (w_k, v), \quad
G(v + \Rw ) \leq G_k (w_k , v).
\end{equation*}
\end{assumption}

\Cref{Ass:stable} is compatible with various stable decomposition conditions presented in existing works, e.g.,~\cite{Badea:2010,TX:2002,TW:2005}.
\Cref{Ass:convex} trivially holds with $\tau_0 = 1/N$ due to the convexity of $E$.
However, a better value for $\tau_0$ independent of $N$ can be found by the usual coloring technique; see~\cite[section~5.1]{Park:2020} for details.
In the same spirit as~\cite{TW:2005}, \cref{Ass:local} gives a one-sided measure of approximation properties of the local solvers.
It was shown in~\cite[section~4.1]{Park:2020} that the above assumptions reduce to~\cite[Assumptions~2.2 to~2.4]{TW:2005}  if they are applied to linear elliptic problems.
Under the above three assumptions, we have the following convergence theorem for Algorithm~\ref{Alg:ASM}~\cite[Theorem~4.7]{Park:2020}.

\begin{proposition}
\label{Prop:conv}
Suppose that \cref{Ass:stable,Ass:convex,Ass:local} hold.
In \cref{Alg:ASM}, we have
\begin{equation*}
E(u^{(n)}) - E(u^*) = O \left( \frac{\kappa_{\tau, \omega}}{n^{q-1}} \right), 
\end{equation*}
where $\kappa_{\tau, \omega}$ is the additive Schwarz condition number defined by
\begin{equation}
\label{kASM}
\kappa_{\tau, \omega} = \frac{\omega C_{0, K_{\tau}}^q }{\tau^{q-1}},
\end{equation}
and $K_{\tau}$ was defined in~\cref{Ktau}.
\end{proposition}

Meanwhile, the {\L}ojasiewicz inequality holds in many applications~\cite{BDL:2007,XY:2013}; it says that the energy functional $E$ of~\cref{model} is sharp around the minimizer $u^*$. 
We summarize this property in \cref{Ass:sharp}; it is well-known that improved convergence results for first-order optimization methods can be obtained under this assumption~\cite{BST:2014,RD:2020}.

\begin{assumption}[sharpness]
\label{Ass:sharp}
There exists a constant $p > 1$ such that for any bounded and convex subset $K$ of $V$ satisfying $u^* \in K$, we have
\begin{equation*}
\frac{\mu_K}{p} \| u - u^* \|^{p} \leq E(u) - E(u^*), \quad u \in K,
\end{equation*}
for some $\mu_K > 0$.
\end{assumption}

We present an improved convergence result for \cref{Alg:ASM} compared to \cref{Prop:conv} under the additional sharpness assumption on $E$~\cite[Theorem~4.8]{Park:2020}.

\begin{proposition}
\label{Prop:conv_sharp}
Suppose that \cref{Ass:stable,Ass:convex,Ass:local,Ass:sharp} hold.
In \cref{Alg:ASM}, we have
\begin{equation*}
E(\un) - E(u^*) =  \begin{cases} O \left( \left(1 - \left( 1- \frac{1}{q} \right) \min \left\{ \tau, \left( \frac{\mu}{q \kappa_{\tau, \omega}} \right)^{\frac{1}{q-1}} \right\} \right)^n \right), & \textrm{ if } p = q, \\ 
O \left( \frac{\left( \kappa_{\tau, \omega}^p / \mu^q \right)^{\frac{1}{p-q}}}{n^{\frac{p(q-1)}{p-q}}} \right), & \textrm{ if } p > q, \end{cases}
\end{equation*}
where $\kappa_{\tau, \omega}$ was defined in~\cref{kASM}.
\end{proposition}

\Cref{Prop:conv,Prop:conv_sharp} are direct consequences of \cref{Lem:ASM} in the sense that they can be easily deduced by invoking theories of gradient methods for convex optimization~\cite[section~2]{Park:2020}.

\section{Backtracking strategies}
\label{Sec:Backtracking}
In gradient methods, backtracking strategies are usually adopted to find a suitable step size that ensures sufficient decrease of the energy.
For problems of the form~\cref{model}, backtracking strategies are necessary in particular to obtain the global convergence to a solution when the Lipschitz constant of $F'$ is not known~\cite{Armijo:1966,BT:2009}.
Considering \cref{Alg:ASM}, a sufficient decrease condition of the energy is satisfied whenever $\tau \in (0, \tau_0]$ and $\omega \geq \omega_0$~(see~\cite[Lemma~4.6]{Park:2020}), and the values of $\tau_0$ and $\omega_0$ in \cref{Ass:convex,Ass:local}, respectively, can be obtained explicitly in many cases.
Indeed, an estimate for $\tau_0$ independent of $N$ can be obtained by the coloring technique~\cite[section~5.1]{Park:2020}, and we have $\omega_0 = 1$ when we use the exact local solvers.
Therefore, backtracking strategies are not essential for the purpose of ensuring the global convergence of additive Schwarz methods.
In this perspective, to the best of our knowledge, there have been no considerations on applying backtracking strategies in the existing works on additive Schwarz methods for convex optimization.

Meanwhile, in several recent works on accelerated first-order methods for convex optimization~\cite{CC:2019,Nesterov:2013,SGB:2014}, full backtracking strategies that allow for adaptive increasing and decreasing of the estimated step size along the iterations were considered.
While classical one-sided backtracking strategies~(see, e.g.,~\cite{BT:2009}) are known to suffer from degradation of the convergence rate if an inaccurate estimate for the step size is computed, full backtracking strategies can be regarded as acceleration schemes in the sense that a gradient method equipped with full backtracking outperforms the method with the known Lipschitz constant~\cite{CC:2019,SGB:2014}.

In this section, we deal with a backtracking strategy for additive Schwarz methods as an acceleration scheme.
Existing full backtracking strategies~\cite{CC:2019,Nesterov:2013,SGB:2014} mentioned above cannot be applied directly to additive Schwarz methods because the evaluation of the nonlinear distance function $M_{\tau, \omega}(\cdot, \cdot)$ is not straightforward due to its complicated definition~(see \cref{Lem:ASM}).
Instead, we propose a novel backtracking strategy for additive Schwarz methods, in which the computational cost of the backtracking procedure is insignificant compared to that of solving local problems.
The abstract additive Schwarz method equipped with the proposed backtracking strategy is summarized in \cref{Alg:back}.

\begin{algorithm}[]
\caption{Additive Schwarz method for~\cref{model} with backtracking}
\begin{algorithmic}[]
\label{Alg:back}
\STATE Choose $u^{(0)} \in \dom G$, $\tau^{(0)}  = \tau_0$, $\omega \geq \omega_0$, and $\rho \in (0, 1)$.
\FOR{$n=0,1,2,\dots$}
\STATE \vspace{-0.5cm} \begin{equation*}
\resizebox{0.9\hsize}{!}{ $\displaystyle
w_k^{(n+1)} \in \argmin_{w_k \in V_k} \left\{ F(\un) + \langle F'(\un), \Rw \rangle + \omega d_k (w_k, \un) + G_k (w_k, \un) \right\}, \gap 1 \leq k \leq N
$}
\end{equation*}
\vspace{-1cm}
\STATE \begin{equation*}
\tau \leftarrow \taun / \rho
\end{equation*}
\vspace{-0.4cm}
\REPEAT
\STATE \vspace{-1cm} \item \begin{equation*}
\unn = \un + \tau \sumk \Rw^{(n+1)}
\end{equation*}
\vspace{-0.2cm}
\IF{$\displaystyle E(\unn) > (1- \tau N) E(\un) + \tau \sumk E (\un + R_k^* w_k^{(n+1)} )$}
\STATE \vspace{-0.3cm} \begin{equation*}
\tau \leftarrow \rho \tau
\end{equation*}
\ENDIF
\UNTIL{$\displaystyle E(\unn) \leq (1- \tau N) E(\un) + \tau \sumk E(\un + R_k^* w_k^{(n+1)} )$}
\STATE \vspace{-0.3cm} \begin{equation*}
\taunn= \tau
\end{equation*}
\ENDFOR
\end{algorithmic}
\end{algorithm}

The parameter $\rho \in (0,1)$ in \cref{Alg:back} plays a role of an adjustment parameter for the grid search.
As $\rho$ closer to $0$, the grid for line search of $\tau$ becomes sparser.
On the contrary, the greater $\rho$, the greater $\tau^{(n+1)}$ is found with the more computational cost for the backtracking process.
The condition $\tau^{(0)} = \tau_0$ is not critical in the implementation of \cref{Alg:back} since $\tau_0$ can be obtained by the coloring technique.

Different from the existing approaches~\cite{CC:2019,Nesterov:2013,SGB:2014}, the backtracking scheme in \cref{Alg:back} does not depend on the distance function $M_{\tau, \omega}(\cdot, \cdot)$ but the energy functional $E$ only.
Hence, the stop criterion
\begin{equation}
\label{back_stop}
E(\unn) \leq  (1 - \tau N) E(\un) + \tau \sumk E(\un + \Rw^{(n+1)})
\end{equation}
for the backtracking process can be evaluated without considering to solve the infimum in the definition~\cref{M} of $M_{\tau, \omega}(\cdot, \cdot)$.
Moreover, the backtracking process is independent of local problems~\cref{local}.
That is, the stop criterion~\cref{back_stop} is universal for any choices of $d_k$ and $G_k$.

The additional computational cost of \cref{Alg:back} compared to \cref{Alg:ASM} comes from the backtracking process.
When we evaluate the stop criterion~\cref{back_stop}, the values of $E(\unn)$, $E(\un)$, and $E(\un + \Rw^{(n+1)})$ are needed.
Among them, $E(\un)$ and $E(\un + \Rw^{(n+1)})$ can be computed prior to the backtracking process since they require $\un$ and $\Rw^{(n+1)}$ only in their computations.
Hence, the computational cost of an additional inner iteration of the backtracking process consists of the computation of $E(\unn)$ only, which is clearly marginal.
In conclusion, the most time-consuming part of each iteration of \cref{Alg:back} is to solve local problems on $V_k$, i.e., to obtain $w_k^{(n+1)}$, and the other part has relatively small computational cost.
This highlights the computational efficiency of the backtracking process in \cref{Alg:back}.

Next, we analyze the convergence behavior of \cref{Alg:back}.
First, we prove that the backtracking process in \cref{Alg:back} ends in finite steps and that the step size $\taun$ never becomes smaller than a particular value.

\begin{lemma}
\label{Lem:back}
Suppose that \cref{Ass:convex} holds.
The backtracking process in \cref{Alg:back} terminates in finite steps and we have
\begin{equation*}
\taun \geq \tau_0
\end{equation*}
for $n \geq 0$, where $\tau_0$ was given in \cref{Ass:convex}.
\end{lemma}
\begin{proof}
Since \cref{Ass:convex} implies that the stop criterion~\cref{back_stop} is satisfied whenever $\tau \in (0, \tau_0]$, the backtracking process ends if $\tau$ becomes smaller than or equal to $\tau_0$.
Now, take any $n \geq 1$.
If $\taun$ were less than $\tau_0$, say $\taun = \rho^j \tau_0$ for some $j \geq 1$, then $\tau$ in the previous inner iteration is $\rho^{j-1} \tau_0 \leq \tau_0$, so that the backtracking process should have stopped there, which is a contradiction.
Therefore, we have $\tau^{(n)} \geq \tau_0$.
\end{proof}

\Cref{Lem:back} says that \cref{Ass:convex} is a sufficient condition to ensure that $\taunn$ is successfully determined by the backtracking process in each iteration of \cref{Alg:back}.
It is important to notice that $\taun$ is always greater than or equal to $\tau_0$; the step sizes of \cref{Alg:back} are larger than or equal to that of \cref{Alg:ASM}.
Meanwhile, similar to the plain additive Schwarz method, \cref{Alg:back} generates the sequence $\{ \un \}$ whose energy is monotonically decreasing.
Hence, $\{ \un \}$ is contained in $K_0$ defined in~\cref{K0}.

\begin{lemma}
\label{Lem:monotone}
Suppose that \cref{Ass:convex} holds.
In \cref{Alg:back}, the sequence $\{ E(\un)\}$ is decreasing.
\end{lemma}
\begin{proof}
Take any $n \geq 0$.
By the stop criterion~\cref{back_stop} for backtracking and the minimization property of $w_k^{(n+1)}$, we get
\begin{equation*}
E(\unn) \leq ( 1 - \taunn N) E(\un) + \taunn \sumk E( \un + \Rw^{(n+1)}) \leq E(\un),
\end{equation*}
which completes the proof.
\end{proof}

Note that~\cite[Lemma~4.6]{Park:2020} played a key role in the convergence analysis of \cref{Alg:ASM} presented in~\cite{Park:2020}.
Relevant results for \cref{Alg:back} can be obtained in a similar manner.

\begin{lemma}
\label{Lem:M_lower}
Suppose that \cref{Ass:convex,Ass:local} hold.
In \cref{Alg:back}, we have
\begin{equation*}
\begin{split}
&D_F(\unn, \un) + G(\unn) \leq M_{\taunn, \omega} (\unn, \un)
\end{split}
\end{equation*}
for $n \geq 0$, where the functional $M_{\tau, \omega}(\cdot, \cdot)$ and the set $K_{\tau}$ were defined in~\cref{M,Ktau}, respectively, for $\tau, \omega > 0$.
\end{lemma}
\begin{proof}
Take any $w_k \in V_k$ such that
\begin{equation}
\label{sufficient_decomposition}
\unn - \un = \taunn \sumk \Rw.
\end{equation}
By \cref{Ass:local} and~\cref{back_stop}, we get
\begin{equation*}
\resizebox{\hsize}{!}{ $\displaystyle
\begin{split}
&\taunn \sumk (\omega d_k + G_k ) (w_k, \un) + (1 - \taunn ) G (\un) \\
&\geq ( 1 - \taunn N) E(\un) + \taunn \sumk E( \un + \Rw ) - F(\un) - \langle F' (\un), \unn - \un \rangle \\
&\geq D_F (\unn, \un) + G(\unn).
\end{split}
$}
\end{equation*}
Taking the infimum over all $w_k$ satisfying~\cref{sufficient_decomposition} yields the desired result.
\end{proof}

\begin{lemma}
\label{Lem:M_upper}
Suppose that \cref{Ass:stable} holds.
Let $\tau, \omega > 0$.
For any bounded and convex subset $K$ of $V$, we have
\begin{equation}
\label{M_upper}
M_{\tau, \omega} (u, v) \leq \frac{\omega C_{0,K_{\tau}'}^q}{q \tau^{q-1}}  \| u - v \|^q + \tau G \left( \frac{1}{\tau} u - \left( \frac{1}{\tau} - 1 \right) v \right) + (1 - \tau ) G(v)
\end{equation}
for $u,v \in K \cap \dom G$, where the functional $M_{\tau, \omega}(\cdot, \cdot)$ was given in~\cref{M} and
\begin{equation*}
K_{\tau}' = \left\{ \frac{1}{\tau} u - \left( \frac{1}{\tau} - 1 \right) v : u,v \in K \right\}.
\end{equation*}
In addition, the right-hand side of~\cref{M_upper} is decreasing with respect to $\tau$.
More precisely, if $\tau_1 \geq \tau_2 > 0$, then we have
\begin{multline}
\label{M_upper2}
\frac{\omega C_{0,K_{\tau_1}'}^q}{q \tau_1^{q-1}}  \| u - v \|^q + \tau_1 G \left( \frac{1}{\tau_1} u - \left( \frac{1}{\tau_1} - 1 \right) v \right) + (1 - \tau_1 ) G(v) \\
\leq \frac{\omega C_{0,K_{\tau_2}'}^q}{q \tau_2^{q-1}}  \| u - v \|^q + \tau_2 G \left( \frac{1}{\tau_1} u - \left( \frac{1}{\tau_2} - 1 \right) v \right) + (1 - \tau_2 ) G(v)
\end{multline}
for $u,v \in K \cap \dom G$.
\end{lemma}
\begin{proof}
\Cref{M_upper} is identical to the second half of~\cite[Lemma~4.6]{Park:2020}.
Nevertheless, it is revisited to highlight that some assumptions given in~\cite[Lemma~4.6]{Park:2020} are not necessary for \cref{Lem:M_upper}; for example, $\tau$ need not be less than or equal to $\tau_0$ as stated in~\cite[Lemma~4.6]{Park:2020} but can be any positive real number.

Now, we prove~\cref{M_upper2}.
Since $\tau_1 \geq \tau_2$, one can deduce from~\cref{Ktau} that $K_{\tau_1}' \subseteq K_{\tau_2}'$.
Hence, by the definition of $C_{0, K}$ given in \cref{Ass:stable}, we get $C_{0, K_{\tau_1}'} \leq C_{0, K_{\tau_2}'}$.
Meanwhile, the convexity of $G$ implies that
\begin{equation*}
\resizebox{\hsize}{!}{ $\displaystyle
\tau_1 G \left( \frac{1}{\tau_1} u - \left( \frac{1}{ \tau_1} -1 \right) v \right) + ( 1 - \tau_1) G(v)
\leq \tau_2 G \left( \frac{1}{\tau_2} u - \left( \frac{1}{ \tau_2 } -1 \right) v \right) + ( 1 - \tau_2) G(v),
$}
\end{equation*}
which completes the proof.
\end{proof}

Recall that the sequence $\{ \tau^{(n)} \}$ generated by \cref{Alg:back} has a uniform lower bound $\tau_0$ by \cref{Lem:back}.
Hence, for any $n \geq 0$, we get
\begin{equation}
\label{starting_point}
\begin{split}
E (\unn) &= F( \un) + \langle F'(\un), \unn - \un \rangle + D_F (\unn, \un)  + G(\unn) \\
&\stackrel{\textrm{(i)}}{\leq} F( \un) + \langle F'(\un), \unn - \un \rangle + M_{\taunn, \omega} (\unn, \un) \\
&\stackrel{\textrm{(ii)}}{=} \min_{u \in K_0} \left\{ F( \un) + \langle F'(\un), u - \un \rangle + M_{\taunn, \omega} (u, \un) \right\} \\
&\stackrel{\textrm{(iii)}}{\leq} \min_{u \in K_0} \Bigg\{ F( \un) + \langle F'(\un), u - \un \rangle
+ \frac{\omega C_{0, K_{\tau_0}}^q}{q \tau_0^{q-1}} \| u - \un \|^q \\
&\quad\quad\quad\quad + \tau_0 G \left( \frac{1}{\tau_0} u - \left( \frac{1}{ \tau_0} -1 \right) \un \right) + ( 1 - \tau_0 ) G (\un) \Bigg\},
\end{split}
\end{equation}
where (i), (ii), and (iii) are because of \cref{Lem:M_lower,Lem:ASM,Lem:M_upper}, respectively.
Starting from~\cref{starting_point}, we readily obtain the following convergence theorems for \cref{Alg:back} by proceeding in the same manner as in~\cite[Appendices~A.3 and~A.4]{Park:2020}.

\begin{proposition}
\label{Prop:naive}
Suppose that \cref{Ass:stable,Ass:convex,Ass:local} hold.
In \cref{Alg:back}, we have
\begin{equation*}
E(u^{(n)}) - E(u^*) = O \left( \frac{\kappa_{\tau_0, \omega}}{n^{q-1}} \right), 
\end{equation*}
where $\kappa_{\tau_0, \omega}$ was defined in~\cref{kASM}.
\end{proposition}

\begin{proposition}
\label{Prop:naive_sharp}
Suppose that \cref{Ass:stable,Ass:convex,Ass:local,Ass:sharp} hold.
In \cref{Alg:back}, we have
\begin{equation*}
E(\un) - E(u^*) =  \begin{cases} O \left( \left(1 - \left( 1- \frac{1}{q} \right) \min \left\{ \tau, \left( \frac{\mu}{q \kappa_{\tau_0, \omega}} \right)^{\frac{1}{q-1}} \right\} \right)^n \right), & \textrm{ if } p = q, \\ 
O \left( \frac{\left( \kappa_{\tau_0, \omega}^p / \mu^q \right)^{\frac{1}{p-q}}}{n^{\frac{p(q-1)}{p-q}}} \right), & \textrm{ if } p > q, \end{cases}
\end{equation*}
where $\kappa_{\tau_0, \omega}$ was defined in~\cref{kASM}.
\end{proposition}

Although \cref{Prop:naive,Prop:naive_sharp} guarantee the convergence to the energy minimum as well as they provide the order of convergence of \cref{Alg:back}, they are not fully satisfactory results in the sense that they are not able to explain why \cref{Alg:back} achieves faster convergence that \cref{Alg:ASM}.
In order to explain the acceleration property of the backtracking process, one should obtain an estimate for the convergence rate of \cref{Alg:back} in terms of the step sizes $\{ \taun \}$ along the iterations~\cite{CC:2019}.
We first state an elementary lemma that will be used in further analysis of \cref{Alg:back}~(cf.~\cite[Lemma~3.2]{TX:2002}).

\begin{lemma}
\label{Lem:recur}
Suppose that $a, b>0$ satisfy the inequality
\begin{equation*}
a - b \geq C a^{\gamma},
\end{equation*}
where $C> 0$ and $\gamma > 1$.
Then we have
\begin{equation*}
b \leq (C (\gamma - 1) + a^{1-\gamma})^{\frac{1}{1-\gamma}} < a.
\end{equation*}
\end{lemma}
\begin{proof}
It suffices to show that $ ( C (\gamma - 1) + a^{1-\gamma})^{\frac{1}{1-\gamma}} \geq a - C a^{\gamma}$.
We may assume that $a - C a^{\gamma} > 0$.
By the mean value theorem, there exists a constant $c \in (a - Ca^{\gamma}, a)$ such that
\begin{multline*}
C(\gamma - 1) + a^{1-\gamma} - (a - Ca^{\gamma})^{1-\gamma} \\
= C(\gamma - 1) + Ca^{\gamma} (1-\gamma) c^{-\gamma}
= C(\gamma - 1) \left( 1 - \left( \frac{a}{c} \right)^{\gamma} \right) \leq 0.
\end{multline*}
Hence, we have $C(\gamma - 1)  + a^{1-\gamma} \leq (a-Ca^{\gamma})^{1-\gamma}$, which yields the desired result.
\end{proof}

We also need the following lemma that was presented in~\cite[Lemma~A.2]{Park:2020}.

\begin{lemma}
\label{Lem:min}
Let $a,b > 0$, $q > 1$, and $\theta > 0$.
The minimum of the function $g(t) = \frac{a}{q}t^q - bt$, $t \in [0, \theta]$ is given as follows:
\begin{equation*}
\min_{t \in [0, \theta]} g(t) = \begin{cases}
\frac{a}{q}\theta^q - b \theta & \textrm{ if } a \theta^{q-1} - b < 0, \\
-\frac{b (q-1)}{q} \left( \frac{b}{a} \right)^{\frac{1}{q-1}} & \textrm{ if } a \theta^{q-1} -b \geq 0. \end{cases}
\end{equation*}
\end{lemma}

Now, we present a convergence theorem for \cref{Alg:back} that reveals the dependency of the convergence rate on the step sizes $\{ \taun \}$ determined by the backtracking process.
More precisely, the following theorems show that the convergence rate of \cref{Alg:back} is dependent on the \textit{$\nu$-averaged additive Schwarz condition number} $\bar{\kappa}_{\nu}$ defined by
\begin{equation}
\label{kASM_averaged}
\bar{\kappa}_{\nu} = \left( \frac{1}{n} \sum_{j=1}^n \kappa_{\tau^{(j)}, \omega}^{\nu} \right)^{\frac{1}{\nu}},
\end{equation}
where $\nu = - \frac{1}{q-1}$ and $\kappa_{\tau^{(j)},\omega}$ was defined in~\cref{kASM}.

\begin{theorem}
\label{Thm:conv}
Suppose that \cref{Ass:stable,Ass:convex,Ass:local} hold.
In \cref{Alg:back}, if $\zeta_0 := E(u^{(0)}) - E(u^*) \leq \omega C_{0, K_0}^q R_0^q$, then
\begin{equation*}
E(\un) - E(u^*)  \leq \frac{q^{q-1}R_0^q \bar{\kappa}_{\nu}}{\left( n + q \left( R_0^q \bar{\kappa}_{\nu} \middle/ \zeta_0 \right)^{-\nu} \right)^{q-1}}
\end{equation*}
for $n \geq 1$, where $K_0$, $R_0$, and $\bar{\kappa}_{\nu}$ were given in~\cref{K0,R0,kASM_averaged}, respectively.
\end{theorem}
\begin{proof}
We take any $n \geq 0$ and write $\zeta_n = E(\un) - E(u^*)$.
For $u \in K_0$, we write
\begin{equation*}
\tilde{u} = \frac{1}{\taunn} u - \left( \frac{1}{\taunn} - 1 \right) \un,
\end{equation*}
so that $u - \un = \taunn (\tilde{u} - \un)$.
It follows that
\begin{equation}
\label{conv_proof1}
\begin{split}
E&(\unn) \stackrel{\cref{starting_point}}{\leq} \min_{u \in K_0} \left\{ F(\un) + \langle F'(\un), u - \un \rangle + M_{\taunn, \omega} (u, \un) \right\} \\
&\stackrel{\textrm{(i)}}{\leq} \min_{u \in K_0} \Bigg\{ F(\un) + \taunn \langle F'(\un), \tilde{u} - \un \rangle + \frac{\taunn \omega C_{0,K_{\taunn}}^q}{q} \| \tilde{u} - \un \|^q \\
&\quad\quad\quad\quad + \taunn G(\tilde{u}) + (1-\taunn) G(\un) \Bigg\} \\
&\stackrel{\textrm{(ii)}}{\leq} \min_{u \in K_0} \left\{ (1-\taunn) E(\un) + \taunn E(\tilde{u}) + \frac{\taunn \omega C_{0,K_{\taunn}}^q}{q} \| \tilde{u} - \un \|^q \right\},
\end{split}
\end{equation}
where (i) is due to \cref{Lem:M_upper} and (ii) is due to the convexity of $F$.
If we set $u = t u^* + (1-t) \un$ for $t \in [0, \taunn] $, then $u \in K_0$ and
\begin{equation}
\label{tu}
\tilde{u} = \frac{t}{\taunn}u^* + \left( 1 - \frac{t}{\taunn} \right) \un \in K_0.
\end{equation}
Substituting~\cref{tu} into~\cref{conv_proof1} yields
\begin{equation}
\label{conv_proof2}
\resizebox{\hsize}{!}{ $\displaystyle \begin{split}
E (\unn) &\leq \min_{t \in [0, \taunn]} \Bigg\{ ( 1- \taunn) E(\un) + \taunn E \left( \frac{t}{\taunn}u^* + \left( 1 - \frac{t}{\taunn} \right) \un \right) \\
&\quad\quad\quad\quad\quad\quad + \frac{\omega C_{0, K_{\taunn}}^q t^q}{q (\taunn)^{q-1}} \| u^* - \un \|^q \Bigg\} \\
&\leq \min_{t \in [0, \taunn]} \left\{ E(\un) - t \zeta_n + \frac{\omega C_{0, K_{\taunn}}^q t^q}{q (\taunn)^{q-1}} R_0^q \right\},
\end{split}
$}
\end{equation}
where the last inequality is due to the convexity of $E$ and~\cref{R0}.
The definition~\cref{Ktau} of $K_{\taunn}$ implies that $K_0 \subseteq K_{\taunn}$, so that $C_{0, K_0} \leq C_{0, K_{\taunn}}$.
Hence, we have $\zeta_n \leq \omega C_{0, K_{\taunn}}^q R_0^q$.
Invoking~\cref{Lem:min}, we get
\begin{equation}
\label{conv_proof3}
\min_{t \in [0, \taunn]} \left\{  - t \zeta_n + \frac{\omega C_{0, K_{\taunn}}^q t^q}{q (\taunn)^{q-1}} R_0^q \right\} = - \frac{q-1}{q} \frac{1}{(\kappa_{\taunn, \omega} R_0^q)^{\frac{1}{q-1}}} \zeta_n^{\frac{q}{q-1}},
\end{equation} 
where $\kappa_{\taunn, \omega}$ was defined in~\cref{kASM}.
Combining~\cref{conv_proof2,conv_proof3} yields
\begin{equation*}
\zeta_n - \zeta_{n+1} \geq \frac{q-1}{q} \frac{1}{(\kappa_{\taunn, \omega} R_0^q)^{\frac{1}{q-1}}} \zeta^{\frac{q}{q-1}}.
\end{equation*}
By \cref{Lem:recur}, it follows that
\begin{equation}
\label{conv_proof4}
\frac{1}{\zeta_{n+1}^{-\nu}} \geq \frac{1}{\zeta_n^{-\nu}} + \frac{1}{q \left( \kappa_{\taunn, \omega} R_0^q \right)^{-\nu}}.
\end{equation}
Summation of~\cref{conv_proof4} over $0, \dots, n-1$ yields
\begin{equation*}
\frac{1}{\zeta_n^{-\nu}} \geq \frac{1}{\zeta_0^{-\nu}} + \frac{n}{q \left( \bar{\kappa}_{\nu} R_0^q \right)^{-\nu}},
\end{equation*}
or equivalently,
\begin{equation*}
\zeta_n \leq \frac{q^{-\frac{1}{\nu}} R_0^q \bar{\kappa}_{\nu}}{\left( n + q \left( R_0^q \bar{\kappa}_{\nu} / \zeta_0 \right)^{-\nu} \right)^{-\frac{1}{\nu}}},
\end{equation*}
which is the desired result.
\end{proof}

\begin{theorem}
\label{Thm:conv_sharp}
Suppose that \cref{Ass:stable,Ass:convex,Ass:local,Ass:sharp} hold.
In \cref{Alg:back}, we have the following:
\begin{enumerate}
\item In the case $p = q$, we have
\begin{multline}
\label{conv_sharp1}
\frac{E(\unn) - E(u^*)}{E(\un) - E(u^*)} \\
 \leq 1 - \taunn \min \left\{ 1 - \frac{\omega C_{0, K_{\taunn}}^q}{\mu}, \frac{q-1}{q} \left( \frac{\mu}{q \omega C_{0, K_{\taunn}}} \right)^{\frac{1}{q-1}} \right\} 
\end{multline}
for $n \geq 0$.
In particular, if $\mu \leq q \omega C_{0, K_0}^q$, then 
\begin{equation*}
E(\un) - E(u^*)
\leq \left(1 - \frac{q-1}{q} \left( \frac{\mu}{q \bar{\kappa}_{\nu}} \right)^{-\nu} \right)^n ( E(u^{(0)} - E(u^*) )
\end{equation*}
for $n \geq 1$, where $K_0$ and $\bar{\kappa}_{\nu}$ were given in~\cref{K0,kASM_averaged}, respectively.

\item In the case $p > q$, if $\zeta_0 : = E(u^{(0)}) - E(u^*) \leq (p^{\frac{q}{p}} \omega C_{0,K_0}^q / \mu^{\frac{q}{p}})^{\frac{p}{p-q}}$, then
\begin{equation*}
E(\un) - E(u^*) \leq \frac{\left( \dfrac{pq}{p-q} \right)^{\frac{p(q-1)}{p-q}} \left( \dfrac{p^{\frac{q}{p}} \bar{\kappa}_{\nu}}{\mu^{\frac{q}{p}}}\right)^{\frac{p}{p-q}}}{\left( n + \dfrac{pq}{p-q} \left( \dfrac{p^{\frac{q}{p}} \bar{\kappa}_{\nu}}{\mu^{\frac{q}{p}}} \right)^{-\nu} \middle/ \zeta_0^{\frac{p-q}{p(q-1)}}\right)^{\frac{p(q-1)}{p-q}}}
\end{equation*}
for $n \geq 1$, where $K_0$ and $\bar{\kappa}_{\nu}$ were given in~\cref{K0,kASM_averaged}, respectively.
\end{enumerate}
\end{theorem}
\begin{proof}
We again take any $n \geq 0$ and write $\zeta_n = E(\un) - E(u^*)$.
By~\cref{conv_proof2} and \cref{Ass:sharp}, we get
\begin{equation}
\label{conv_proof5}
\resizebox{\hsize}{!}{ $\displaystyle \begin{split}
E (\unn) &\leq \min_{t \in [0, \taunn]} \Bigg\{ ( 1- \taunn) E(\un) + \taunn E \left( \frac{t}{\taunn}u^* + \left( 1 - \frac{t}{\taunn} \right) \un \right) \\
&\quad\quad\quad\quad\quad\quad + \frac{\omega C_{0, K_{\taunn}}^q t^q}{q (\taunn)^{q-1}} \| u^* - \un \|^q \Bigg\} \\
&\leq \min_{t \in [0, \taunn]} \left\{ E(\un) - t \zeta_n + \frac{p^{\frac{q}{p}}\omega C_{0, K_{\taunn}}^q t^q}{q \mu^{\frac{q}{p}} (\taunn)^{q-1}} \zeta_n^{\frac{q}{p}} \right\}.
\end{split}
$}
\end{equation}

First, we consider the case $p = q$. 
\Cref{conv_proof5} reduces to
\begin{equation}
\label{conv_proof6}
\zeta_{n+1} \leq \min_{t \in [0, \taunn]} \left\{ 1 - t + \frac{\omega C_{0, K_{\taunn}}^q t^q}{\mu (\taunn)^{q-1}} \right\} \zeta_n.
\end{equation}
Invoking \cref{Lem:min}, we have
\begin{multline*}
\min_{t \in [0, \taunn]} \left\{- t + \frac{\omega C_{0, K_{\taunn}}^q t^q}{\mu (\taunn)^{q-1}} \right\} \\
\leq - \taunn \min \left\{ 1 - \frac{\omega C_{0, K_{\taunn}}^q}{\mu}, \frac{q-1}{q} \left( \frac{\mu}{q \omega C_{0, K_{\taunn}}} \right)^{\frac{1}{q-1}} \right\},
\end{multline*}
which implies~\cref{conv_sharp1}.
Now, we assume that $\mu \leq q \omega C_{0, K_0}^q$.
Since $\mu \leq q \omega C_{0, K_0}^q \leq q \omega C_{0, K_{\taunn}}^q$,
it follows by \cref{Lem:min} that
\begin{equation}
\label{conv_proof7}
\min_{t \in [0, \taunn]} \left\{- t + \frac{\omega C_{0, K_{\taunn}}^q t^q}{\mu (\taunn)^{q-1}} \right\}
= - \frac{q-1}{q} \left( \frac{\mu}{q \kappa_{\taunn, \omega}} \right)^{\frac{1}{q-1}},
\end{equation}
where $\kappa_{\taunn, \omega}$ was defined in~\cref{kASM}.
Combining~\cref{conv_proof6,conv_proof7}, it readily follows that
\begin{equation}
\label{conv_proof8}
\zeta_{n+1} \leq \left(1 - \frac{q-1}{q} \left( \frac{\mu}{q \kappa_{\taunn, \omega}} \right)^{\frac{1}{q-1}}  \right) \zeta_n.
\end{equation}
Applying~\cref{conv_proof8} recursively, we obtain
\begin{equation*}
\zeta_n\leq \zeta_0 \prod_{j=1}^n \left( 1 - \frac{q-1}{q} \left( \frac{\mu}{q \kappa_{\tau^{(j)}, \omega}} \right)^{-\nu} \right) 
\leq \left(1 - \frac{q-1}{q} \left( \frac{\mu}{q \bar{\kappa}_{\nu}} \right)^{-\nu} \right)^n \zeta_0,
\end{equation*}
where the last inequality is due to the concavity of the logarithmic function.

Next, we consider the case $p > q$; we assume that $\zeta_0 \leq (p^{\frac{q}{p}} \omega C_{0,K_0}^q / \mu^{\frac{q}{p}})^{\frac{p}{p-q}}$.
Using \cref{Lem:min} the fact that $C_{0,K_0} \leq C_{0, K_{\taunn}}$, one can deduce from~\cref{conv_proof5} that
\begin{equation*}
\zeta_n - \zeta_{n+1} \geq \frac{q-1}{q} \left( \frac{\mu^{\frac{q}{p}}}{p^{\frac{q}{p}} \kappa_{\taunn, \omega}} \right)^{\frac{1}{q-1}} \zeta_n^{\frac{q(p-1)}{p(q-1)}}.
\end{equation*}
Invoking \cref{Lem:recur} yields
\begin{equation}
\label{conv_proof9}
\frac{1}{\zeta_{n+1}^{\frac{1}{\beta}}} \geq \frac{1}{\zeta_{n}^{\frac{1}{\beta}}} + \frac{p-q}{pq} \left( \frac{\mu^{\frac{q}{p}}}{p^{\frac{q}{p}} \kappa_{\taunn, \omega}} \right)^{-\nu},
\end{equation}
where $\beta = \frac{p(q-1)}{p-q}$.
Applying~\cref{conv_proof9} recursively, we get
\begin{equation*}
\frac{1}{\zeta_n^{\frac{1}{\beta}}} \geq \frac{1}{\zeta_0^{\frac{1}{\beta}}} + \frac{p-q}{pq} \left( \frac{\mu^{\frac{q}{p}}}{p^{\frac{q}{p}} \bar{\kappa}_{\nu}} \right)^{-\nu} n,
\end{equation*}
which is equivalent to
\begin{equation*}
\zeta_n \leq \frac{\left( \frac{pq}{p-q}\right)^{\beta} \left( \frac{p^{\frac{q}{p}} \bar{\kappa}_{\nu}}{\mu^{\frac{q}{p}}}\right)^{-\beta \nu}}{\left( n + \frac{pq}{p-q} \left( \frac{p^{\frac{q}{p}} \bar{\kappa}_{\nu}}{\mu^{\frac{q}{p}}}\right)^{-\nu} \middle/ \zeta_0^{\frac{1}{\beta}} \right)^{\beta}}.
\end{equation*}
This completes the proof.
\end{proof}

\begin{remark}
\label{Rem:recur}
If one sets $\taun = \tau$ for all $n$ in the proof of \cref{Thm:conv}, then the following estimate for the convergence rate of  \cref{Alg:ASM} is obtained:
\begin{equation*}
E(\un) - E(u^*) \leq \frac{q^{q-1} R_0^q \kappa_{\tau, \omega}}{ \left( n + q \left( R_0^q \kappa_{\tau, \omega} / \zeta_0 \right)^{\frac{1}{q-1}} \right)^{q-1}}.
\end{equation*}
This estimate is asymptotically equivalent to~\cite[Theorem~4.7]{Park:2020}, but differs by a multiplicative constant.
A similar remark can be made for \cref{Thm:conv_sharp} and~\cite[Theorem~4.8]{Park:2020}.
\end{remark}

Similar to the discussions made in~\cite{CC:2019}, \cref{Thm:conv,Thm:conv_sharp} can be interpreted as follows: since the convergence rate of \cref{Alg:back} depends on the averaged quantity~\cref{kASM_averaged}, adaptive adjustment of $\tau$ depending on the local flatness of the energy functional can be reflected to the convergence rate of the algorithm.
As we observed in \cref{Lem:back}, $\tau^{(n)}$ in \cref{Alg:back} is always greater than or equal to $\tau_0$.
Therefore, \cref{Thm:conv,Thm:conv_sharp} imply that \cref{Alg:back} enjoys better convergence rate estimates than \cref{Alg:ASM}.

\section{Further acceleration by momentum}
\label{Sec:Momentum}
In the author's recent work~\cite{Park:2021b}, it was shown that the convergence rate of the additive Schwarz method can be significantly improved if an appropriate momentum acceleration scheme~(see, e.g.,~\cite{BT:2009,Nesterov:2013}) is applied.
More precisely, \cref{Alg:ASM} was integrated with the FISTA~(Fast Iterative Shrinkage-Thresholding Algorithm) momentum~\cite{BT:2009} and the gradient adaptive restarting scheme~\cite{OC:2015} to form an accelerated version of the method; see~\cite[Algorithm~5]{Park:2021b}.

Meanwhile, two acceleration schemes for gradient methods, full backtracking and momentum, are compatible to each other; they can be applied to a gradient method simultaneously without disturbing each other and reducing their accelerating effects. 
Indeed, some notable works on full backtracking~\cite{CC:2019,Nesterov:2013,SGB:2014} considered momentum acceleration of gradient methods with full backtracking.
In this viewpoint, we present an further accelerated variant of \cref{Alg:back} in \cref{Alg:unified}, which is a unification of the ideas from~\cite[Algorithm~5]{Park:2021b} and~\cref{Alg:back}.

\begin{algorithm}[]
\caption{Additive Schwarz method for~\eqref{model} with backtracking and momentum}
\begin{algorithmic}[]
\label{Alg:unified}
\STATE Choose $u^{(0)} = v^{(0)} \in \dom G$, $\tau^{(0)}  =  \tau_0$, $\omega \geq \omega_0$, $\rho \in (0, 1)$, and $t_0 = 1$.
\FOR{$n=0,1,2,\dots$}
\STATE \vspace{-0.5cm} \begin{equation*}
\resizebox{0.9\hsize}{!}{ $\displaystyle
w_k^{(n+1)} \in \argmin_{w_k \in V_k} \left\{ F(\vn) + \langle F'(\vn), \Rw \rangle + \omega d_k (w_k, \vn) + G_k (w_k, \vn) \right\}, \gap 1 \leq k \leq N
$}
\end{equation*}
\vspace{-1cm}
\STATE \begin{equation*}
\tau \leftarrow \taun / \rho
\end{equation*}
\vspace{-0.4cm}
\REPEAT
\STATE \vspace{-1cm} \item \begin{equation*}
\unn = \vn + \tau \sumk \Rw^{(n+1)}
\end{equation*}
\vspace{-0.2cm}
\IF{$\displaystyle E(\unn) > (1- \tau N) E(\vn) + \tau \sumk E (\vn + R_k^* w_k^{(n+1)} )$}
\STATE \vspace{-0.3cm} \begin{equation*}
\tau \leftarrow \rho \tau
\end{equation*}
\ENDIF
\UNTIL{$\displaystyle E(\unn) \leq (1- \tau N) E(\vn) + \tau \sumk E(\vn + R_k^* w_k^{(n+1)} )$}
\STATE \vspace{-0.3cm} \begin{equation*}
\taunn= \tau
\end{equation*}
\vspace{-1cm}
\STATE \begin{equation*}
\begin{cases}
t_{n+1} = 1, \gap
\beta_n = 0,
& \textrm{ if } \langle \vn - \unn, \unn - \un \rangle  > 0, \\
t_{n+1} = \frac{1 + \sqrt{1 + 4t_n^2}}{2}, \gap
\beta_n = \frac{t_n - 1}{t_{n+1}},
& \textrm{ otherwise.} 
\end{cases}
\end{equation*}
\vspace{-0.8cm}
\STATE \begin{equation*}
\vnn = \unn + \beta_n (\unn - \un)
 \end{equation*} 
\ENDFOR
\end{algorithmic}
\end{algorithm}

As mentioned in~\cite{Park:2021b}, a major advantage of the momentum acceleration scheme used in \cref{Alg:unified} is that a priori information on the sharpness of the energy $E$ such as the values of $p$ and $\mu_K$ in \cref{Ass:sharp} is not required.
Such adaptiveness to the properties of the energy has become an important issue on the development of first-order methods for convex optimization; see, e.g.,~\cite{OC:2015,RG:2021}.
Compared to \cref{Alg:back}, the additional computational cost of \cref{Alg:unified} comes from the computation of momentum parameters $t_n$ and $\beta_n$, which is clearly marginal.
Therefore, the main computational cost of each iteration of \cref{Alg:unified} is essentially the same as the one of \cref{Alg:back}.
Nevertheless, we will observe in \cref{Sec:Numerical} that \cref{Alg:unified} achieves much faster convergence to the energy minimum compared to \cref{Alg:ASM,Alg:back}.

For completeness, we present a brief explanation on why \cref{Alg:unified} achieves faster convergence than \cref{Alg:back}; one may refer to~\cite{OC:2015,Park:2021b} for more details.
On the one hand, the recurrence formula
$t_{n+1} = (1 + \sqrt{1 + 4t_n^2})/2$
for the momentum parameter $t_n$ in \cref{Alg:unified} is the same as that in FISTA~\cite{BT:2009}.
Hence, the overrelaxation step $\vnn = \beta_n ( \unn - \un)$ in \cref{Alg:unified} is expected to result acceleration of the convergence by the same principle as in FISTA; see~\cite[Figure~8.5]{GBC:2016} for a graphical description of momentum acceleration.
On the other hand, the restart criterion $\langle \vn - \unn, \unn - \un \rangle > 0$ in \cref{Alg:unified} means that the update direction $\unn - \un$ is on the same side of the $M_{\tau,\omega}$-gradient direction $\vn - \unn$.
In the sense that the energy decreases fastest toward the minus gradient direction, satisfying the restart criterion implies that the overrelaxation step was not beneficial, so that we reset the overrelaxation parameter $\beta_n$ as $0$.
In view of dynamical systems, it was observed in~\cite{OC:2015} that the restarting scheme used in \cref{Alg:unified} prevents underdamping of a dynamical system representing the algorithm, so that oscillations of the energy do not occur.

\section{Numerical results}
\label{Sec:Numerical}
In order to show the computational efficiency of \cref{Alg:back,Alg:unified}, we present numerical results applied to various convex optimization problems.
As in~\cite{Park:2021b}, the following three model problems are considered: $s$-Laplace equation~\cite{TX:2002} with two-level domain decomposition, obstacle problem~\cite{BTW:2003,Tai:2003,THX:2002} with two-level domain decomposition, and dual total variation~(TV) minimization~\cite{CTWY:2015,Park:2021a} with one-level domain decomposition.
All the details such as problem settings, finite element discretization, space decomposition, stop criteria for local and coarse problems, and initial parameter settings for the algorithms are set in the same manner as in~\cite[section~4]{Park:2021b} unless otherwise stated, so that we omit them.
We set the fine mesh size $h$, coarse mesh size $H$, and overlapping width $\delta$ among subdomains by $h=1/2^6$, $H = 1/2^3$, and $\delta/h = 4$, respectively, in all experiments.

\begin{figure}[]
\centering
\subfloat[][$s$-Laplace equation]{ \includegraphics[width=0.31\linewidth]{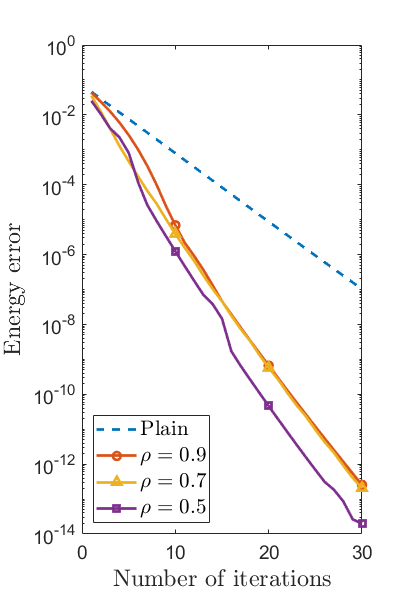} }
\subfloat[][Obstacle problem]{ \includegraphics[width=0.31\linewidth]{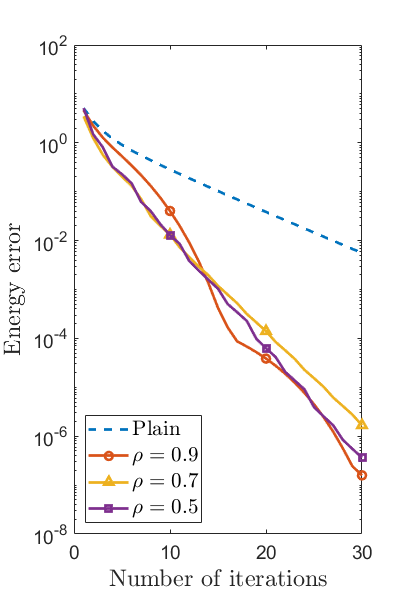} }
\subfloat[][Dual TV minimization]{ \includegraphics[width=0.31\linewidth]{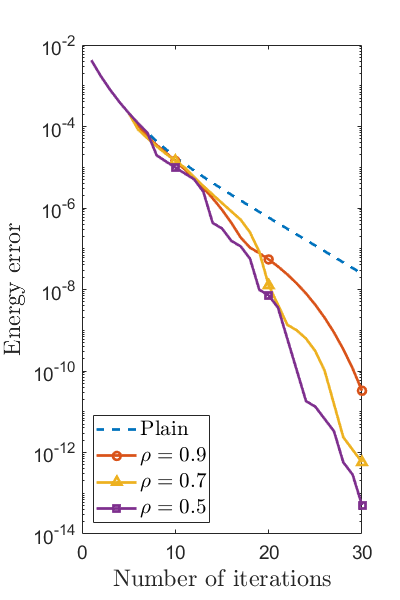} }
\caption{Decay of the energy error $E(\un) - E(u^*)$ in \cref{Alg:back} with respect to various $\rho$~($h = 1/2^6$, $H = 1/2^3$, $\delta = 4h$).
``Plain'' denotes \cref{Alg:ASM}.}
\label{Fig:rho}
\end{figure}

\begin{figure}[]
\centering
\subfloat[][$s$-Laplace equation]{ \includegraphics[width=0.31\linewidth]{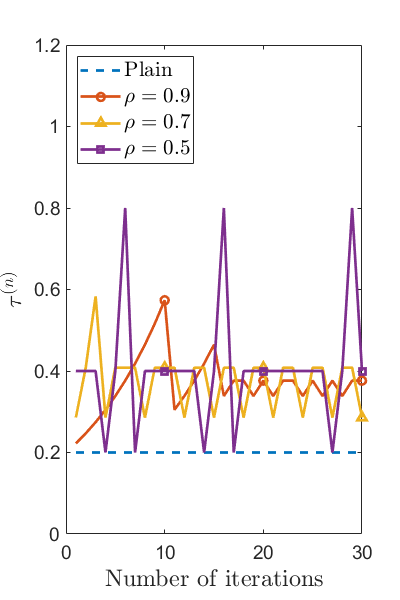} }
\subfloat[][Obstacle problem]{ \includegraphics[width=0.31\linewidth]{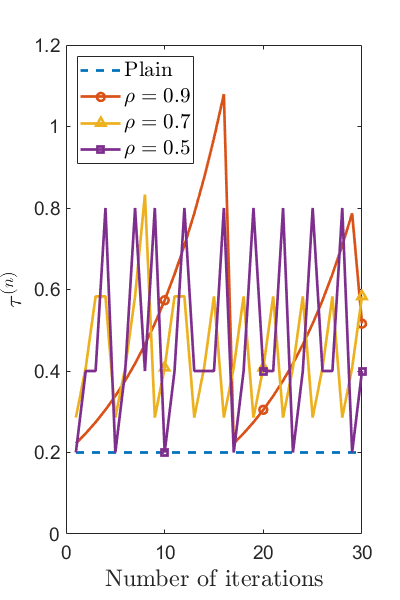} }
\subfloat[][Dual TV minimization]{ \includegraphics[width=0.31\linewidth]{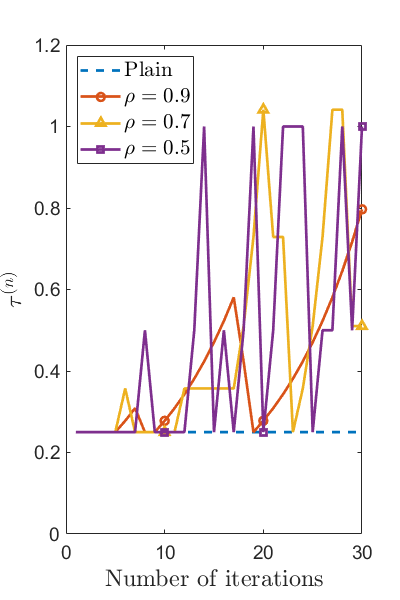} }
\caption{Step sizes $\{ \tau^{(n)} \}$ generated by \cref{Alg:back} with respect to various $\rho$~($h = 1/2^6$, $H = 1/2^3$, $\delta = 4h$).
``Plain'' denotes \cref{Alg:ASM}, which has the constant step size $\tau = \tau_0$.}
\label{Fig:tau}
\end{figure}

First, we observe how the choices of the adjustment parameter $\rho$ affect on the convergence behavior of \cref{Alg:back}.
\Cref{Fig:rho} plots the energy error $E(\un) - E(u^*)$ of \cref{Alg:ASM} and \cref{Alg:back} with $\rho \in \{ 0.5, 0.7, 0.9 \}$.
For every model problem and every value of $\rho$, \cref{Alg:back} shows faster convergence to the energy minimum than \cref{Alg:ASM}.
That is, \cref{Alg:back} outperforms \cref{Alg:ASM} regardless of the choice of $\rho$ in the sense of the convergence rate.
Indeed, as shown in \cref{Fig:tau}, the step sizes $\{ \tau^{(n)} \}$ generated by \cref{Alg:back} always exceed the step size $\tau_0$ of \cref{Alg:ASM}.
Hence, \cref{Fig:tau} verifies \cref{Lem:back}, and faster convergence of  \cref{Alg:back} can be explained by \cref{Thm:conv,Thm:conv_sharp}.
Meanwhile, \cref{Fig:rho,Fig:tau} do not show a clear pattern on the convergence rate of \cref{Alg:back} with respect to $\rho$.
It would be interesting to find a theoretically optimal $\rho$ that results in the fastest convergence rate of \cref{Alg:back}, which is left as a future work.

\begin{figure}[]
\centering
\subfloat[][$s$-Laplace equation]{ \includegraphics[width=0.31\linewidth]{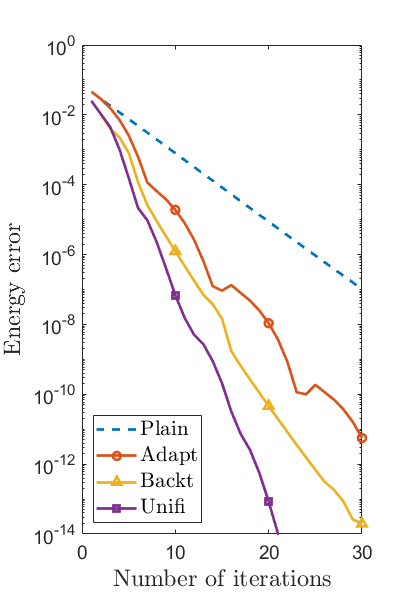} }
\subfloat[][Obstacle problem]{ \includegraphics[width=0.31\linewidth]{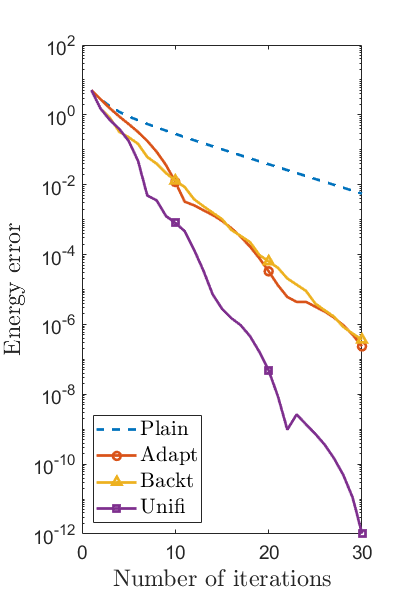} }
\subfloat[][Dual TV minimization]{ \includegraphics[width=0.31\linewidth]{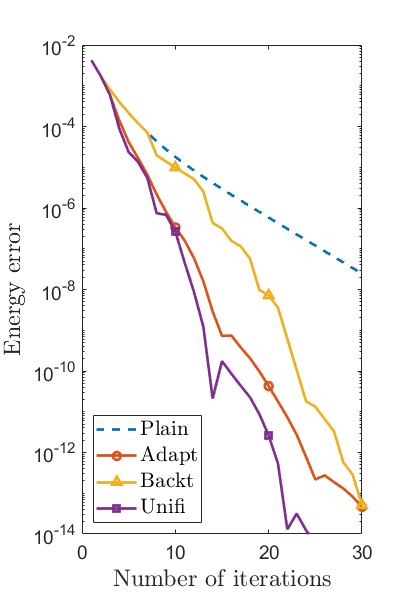} }
\caption{Decay of the energy error $E(\un) - E(u^*)$ in various additive Schwarz methods~($h = 1/2^6$, $H = 1/2^3$, $\delta = 4h$).
``Plain'', ``Adapt'', ``Backt'', and ``Unifi'' denote \cref{Alg:ASM}, Algorithm~5 of~\cite{Park:2021b}, \cref{Alg:back}~($\rho = 0.5$), and \cref{Alg:unified}~($\rho = 0.5$), respectively.}
\label{Fig:comparison}
\end{figure}

Next, we compare the performance of various additive Schwarz methods considered in this paper.
\Cref{Fig:comparison} plots $E(\un) - E(u^*)$ of \cref{Alg:ASM}~(Plain), Algorithm~5 of~\cite{Park:2021b}~(Adapt), \cref{Alg:back} with $\rho = 0.5$~(Backt), and \cref{Alg:unified} with $\rho = 0.5$~(Unifi).
In each of the model problems, the performance of Backt seems similar to that of Adapt.
More precisely, Backt outperforms Adapt in the $s$-Laplace problem, shows almost the same convergence rate as Adapt in the obstacle problem, and shows a bit slower energy decay than Adapt in the first several iterations but eventually arrive at the comparable energy error in the dual TV minimization.
Hence, we can say that the acceleration performance of Backt is comparable to that of Adapt.
Meanwhile, as we considered in \cref{Sec:Momentum}, acceleration schemes used by Adapt and Backt are totally different to each other, and they can be combined to form a further accelerated method Unifi.
One can observe in \cref{Fig:comparison} that Unifi shows the fastest convergence rate among all the methods for every model problem.
Since the difference between the computational costs of a single iteration of Plain and Unifi is insignificant, we can conclude that Unifi possesses the best computational efficiency among all the methods, absorbing the advantages of Adapt and Backt.

\section{Conclusion}
\label{Sec:Conclusion}
In this paper, we proposed a novel backtracking strategy for the additive Schwarz method for the general convex optimization.
It was proven rigorously that the additive Schwarz method with backtracking achieves faster convergence rate than the plain method.
Moreover, we showed that the proposed backtracking strategy can be combined with the momentum acceleration technique proposed in~\cite{Park:2021b}, and proposed a further accelerated additive Schwarz method, \cref{Alg:unified}.
Numerical results verifying our theoretical results and the superiority of the proposed methods were presented.

We observed in \cref{Sec:Numerical} that the additive Schwarz method with backtracking achieves faster convergence behavior than the plain method for any choice of the adjustment parameter $\rho$.
However, it remains as an open problem that what value of $\rho$ results the fastest convergence rate.
Optimizing $\rho$ for the sake of construction of a faster additive Schwarz method will be considered as a future work.

\bibliographystyle{siamplain}
\bibliography{refs_ASM_back}
\end{document}